\documentclass[10pt,a4paper]{article}
\usepackage[utf8]{inputenc}
\usepackage{amsmath}
\usepackage{amsfonts}
\usepackage{amssymb}
\usepackage{amsthm}
\usepackage{color}
\usepackage{enumerate}
\usepackage{marvosym}
\usepackage{hyperref}
 \newtheorem{thm}{Theorem}[section]

 \newtheorem*{thm*}{Theorem}
 \newtheorem{cor}[thm]{Corollary}
 \newtheorem{lem}[thm]{Lemma}
 \newtheorem{prop}[thm]{Proposition}
 \theoremstyle{definition}
 
 \theoremstyle{remark}
 \newtheorem{rem}[thm]{Remark}
 \newtheorem*{conj}{Conjecture}

 \numberwithin{equation}{section}
 
\newcommand{\vertiii}[1]{{\left\vert\kern-0.25ex\left\vert\kern-0.25ex\left\vert #1
  \right\vert\kern-0.25ex\right\vert\kern-0.25ex\right\vert}}

\newcommand{\im}{\operatorname{Im}}
\newcommand{\re}{\operatorname{Re}}

\newcommand{\VERT}[1]{{\left\vert\kern-0.25ex\left\vert\kern-0.25ex\left\vert #1 
    \right\vert\kern-0.25ex\right\vert\kern-0.25ex\right\vert}}

\begin{document}
\title{Essential positivity for Toeplitz operators on the Fock space}
\author{Robert Fulsche}

\maketitle
\begin{abstract}
In this short note, we discuss essential positivity of Toeplitz operators on the Fock space, as motivated by a recent question of Per\"{a}l\"{a} and Virtanen \cite{Perala_Virtanen2023}. We give a proper characterization of essential positivity in terms of limit operators. A conjectured characterization of essential positivity of Per\"{a}l\"{a} and Virtanen is disproven when the assumption of radiality is dropped. Nevertheless, when the symbol of the Toeplitz operator is of vanishing mean oscillation, we show that the conjecture of Per\"{a}l\"{a} and Virtanen holds true, even without radiality.


\end{abstract}

\section{Introduction}

In their recent note \cite{Perala_Virtanen2023}, Per\"{a}l\"{a} and Virtanen study the notion of \emph{essential positivity} and describe a class of Toeplitz operators which are essentially positive. A bounded self-adjoint operator $A$ on the Hilbert space $\mathcal H$ is said to be essentially positive provided $\sigma_{ess}(A) \subseteq [0, \infty)$. Here, $\sigma_{ess}(A)$ is the essential spectrum:
\begin{align*}
\sigma_{ess}(A) = \{ \lambda \in \mathbb C: ~A-\lambda \text{ is not Fredholm}\}.
\end{align*}
After discussing some characterizations of essential positivity (\cite[Theorem 3]{Perala_Virtanen2023}), they go on to characterize essential positivity of certain Toeplitz operators on the Hardy space and the Bergman space of the disc. Their result on the Bergman space of the disc is the following:
\begin{thm*}[{\cite[Theorem 9]{Perala_Virtanen2023}}]
Let $\mu$ be a radial real-valued Borel measure on $\mathbb D$ such that $|\mu|$ is a Carleson measure for the Bergman space $A^2(\mathbb D)$. Suppose that the limit $L = \lim_{|z| \to 1} \tilde{\mu}(z)$ exists. Then, $T_\mu$ is essentially positive on $A^2(\mathbb D)$ if and only if $L \geq 0$.
\end{thm*}
Further, Per\"{a}l\"{a} and Virtanen conjectured that essential positivity of Toeplitz operators on the Bergman space of the disc might in general by classified by the property that $\liminf_{|z|\to 1} \widetilde{f}(z) \geq 0$. More specifically, they suggested:
\begin{conj}
Let $f \in L^\infty(\mathbb D)$ be real-valued and radial. Then, $T_f: A^2(\mathbb D) \to A^2(\mathbb D)$ is essentially positive if and only if $\liminf_{|z| \to 1} \widetilde{f}(z) \geq 0$.
\end{conj}

Per\"{a}l\"{a} and Virtanen make crucial use of a certain Tauberian theorem in the proof of their above theorem. At the end of their work, they ask if an analogous result holds true for Toeplitz operators on the Fock space, pointing out that a suitable substitute for the Tauberian theorem needs to be found.

In this short note, we explain how to obtain an analogous result on the Fock space through the methods of limit operators. We prove that their conjecture is (at least when the assumption of radiality is dropped) wrong on the Fock space, but it is true if the symbol is assumed to be of vanishing mean oscillation. We provide a general characterization of essential positivity in terms of limit operators.  As a special case, we obtain a result analogous to that of Per\"{a}l\"{a} and Virtanen but without the assumption of the symbol being radial.

\section{Essential positivity for Toeplitz operators on the Fock space}
On $\mathbb C^n$ we consider the family of probability measures $\mu_t$ given by
\begin{align*}
d\mu_t(z) = \frac{1}{(\pi t)^n} e^{-\frac{|z|^2}{t}}~dz,
\end{align*}
where $| \cdot |$ is the Euclidean norm, $dz$ the standard Lebesgue measure and $t > 0$ a fixed real number. We define the Fock space $F_t^2$ by
\begin{align*}
F_t^2 = L^2(\mathbb C^n, \mu_{t}) \cap \operatorname{Hol}(\mathbb C^n),
\end{align*}
where $\operatorname{Hol}(\mathbb C^n)$ denotes the entire functions on $\mathbb C^n$. The standard reference for the Fock space and its properties is \cite{Zhu2012}. This space is always endowed with its natural inner product, i.e.
\begin{align*}
\langle f, g\rangle = \int_{\mathbb C^n} f(z) \overline{g(z)}~d\mu_t(z).
\end{align*}
$F_t^2$ is well-known to be a reproducing kernel Hilbert space, the reproducing kernels being given by
\begin{align*}
K_z^t(w) = e^{\frac{w \cdot \overline{z}}{t}}.
\end{align*}
The normalized reproducing kernels are now defined as
\begin{align*}
k_z^t(w) = \frac{K_z^t(w)}{\| K_z^t\|_{F_t^2}} = e^{\frac{w \cdot \overline{z}}{t} - \frac{|z|^2}{2t}}.
\end{align*}
For $A \in \mathcal L(F_t^2)$, we define its \emph{Berezin transform} as the function $\widetilde{A}(z) = \langle Ak_z^t, k_z^t\rangle$, where $z \in \mathbb C^n$. Then, $\widetilde{A}$ is always a bounded and uniformly continuous function on $\mathbb C^n$.

We of course have the well-known orthogonal projection $P_t \in \mathcal L(L^2(\mathbb C^n, \mu_t))$ mapping onto $F_t^2$ by
\begin{align*}
P_t f(z) = \langle f, K_z^t\rangle = \int_{\mathbb C^n} f(w) e^{\frac{z \cdot \overline{w}}{t}} ~d\mu_t(w).
\end{align*}
For any $f \in L^\infty(\mathbb C^n)$ the Toeplitz operator $T_f^t$ given by
\begin{align*}
T_f^t: F_t^2 \to F_t^2, \quad T_f^t(g) = P_t (fg)
\end{align*}
is well defined and bounded, satisfying $\| T_f^t\| \leq \| f\|_\infty$. 

Given some signed Borel measure $\nu$ on $\mathbb C^n$, one defines the Toeplitz operator with symbol $\nu$ formally as
\begin{align*}
T_\nu^t f(z) = \frac{1}{(\pi t)^n} \int_{\mathbb C^n} f(w) e^{\frac{z \cdot \overline{w}}{t} - \frac{|w|^2}{t}}~d\nu(w).
\end{align*}
Of course, in general this expression is not well-defined. The class of positive measures which define in this way a bounded linear operator are well-understood. They are usually called \emph{Fock-Carleson measures} and characterized by the following properties:
\begin{thm*}[{\cite[Theorem 3.1]{Isralowitz_Zhu2010}}]
Let $\nu$ be a positive Borel measure on $\mathbb C^n$. Then, the following are equivalent:
\begin{enumerate}[(i)]
\item $T_\mu^t$ is bounded. 
\item For some (equivalently: every) $R > 0$ there exists a constant $C_R > 0$ such that $\nu(B(z,R)) \leq C_R$ for every $z \in \mathbb C^n$.
\item $\widetilde{\mu}(z) := \frac{1}{(\pi t)^n} \int_{\mathbb C^n} e^{-\frac{|z-w|^2}{t}}~d\nu(w)$ is a bounded function.
\end{enumerate}
When these properties are satisfied, $\widetilde{\mu}$ agrees with the Berezin transform of $T_\mu^t$, $\widetilde{\mu}(z) = \langle T_\mu^t k_z^t, k_z^t\rangle$, and $\| T_\mu^t\| \cong \|\widetilde{\mu}\|_\infty$.
\end{thm*}
Now, if $\nu$ is any signed Borel measure on $\mathbb C^n$, we can consider its Hahn-Jordan decomposition
\begin{align*}
\nu = \nu_+ - \nu_-
\end{align*}
and the total variation measure
\begin{align*}
|\nu| = \nu_+ + \nu_-.
\end{align*}
If one assumes that $|\nu|$ is a Carleson measure, then clearly both $\nu_+$ and $\nu_-$ satisfy (ii) in the above theorem. Hence, both $T_{\nu_+}^t$ and $T_{\nu_-}^t$ are bounded operators on $F_t^2$, which yields that their linear combination $T_\nu^t = T_{\nu_+}^t - T_{\nu_-}^t$ is also a bounded linear operator.

For $z \in \mathbb C^n$ we define the Weyl operator $W_z^t$ by
\begin{align*}
W_z^tg(w) = k_z^t(w) g(w-z). 
\end{align*}
Indeed, $W_z^t = T_{g_z^t}^t$ with
\begin{align*}
g_z^t(w) = e^{\frac{|z|^2}{2t} + \frac{2i \operatorname{Im}(w \cdot \overline{z})}{t}}.
\end{align*}
These operators satisfy the following well-known properties, which we fix as a lemma:
\begin{lem}
\begin{enumerate}[(1)]
\item $W_z^t$ is unitary on $F_t^2$;
\item $W_z^t W_w^t = e^{-\frac{\operatorname{Im}(z \cdot \overline w)}{t}}W_{z+w}^t$ for any $z, w \in \mathbb C^n$; 
\item $z \mapsto W_z^t$ is continuous in strong operator topology over $F_t^2$.
\end{enumerate}
\end{lem}
Note that (1) and (2) above follow from elementary computations, while (3) holds by Scheff\'{e}'s lemma.

It has been proven very fruitful in the study of Toeplitz operators on the Fock space to consider the shift action of $\mathbb C^n$ on operators, defined as $\alpha_z(A) = W_z A W_{-z}$ for $A \in \mathcal L(F_t^2)$ and $z \in \mathbb C^n$, cf.~\cite{Fulsche2020, Fulsche2024}. In particular, we can consider the $C^\ast$-algebra
\begin{align*}
\mathcal C_1 = \{ A \in \mathcal L(F_t^2): ~z \mapsto \alpha_z(A) \text{ is } \| \cdot\|_{op} \text{-cont.}\}.
\end{align*}
Then, we have the following important consequence:
\begin{thm}[{\cite[Theorem 3.1]{Fulsche2020}}]
The following equalities hold true:
\begin{align*}
\mathcal C_1 = C^\ast(\{ T_f^t: ~f \in L^\infty(\mathbb C^n)\}) = \overline{\{ T_f^t: ~f \in \operatorname{BUC}(\mathbb C^n)\}}.
\end{align*}
\end{thm}
Here, $\operatorname{BUC}(\mathbb C^n)$ denotes the bounded, uniformly continuous functions on $\mathbb C^n$.

Now, if $\nu$ is a signed Borel measure on $\mathbb C^n$ such that $|\nu|$ is Fock-Carleson, \cite[Theorem 3.7]{Bauer_Isralowitz2012}\footnote{The result of Bauer and Isralowitz also works for complex measures $\nu$ such that $|\nu|$ is Fock-Carleson. Nevertheless, due to the problem we want to discuss, we are only interested in signed measures.} shows that $T_\nu^t$ can be approximated in norm by Toeplitz operators with bounded symbols. We therefore obtain:
\begin{thm}[{\cite[Theorem 3.7]{Bauer_Isralowitz2012}}]
Let $\nu$ be a signed Borel measure on $\mathbb C^n$ such that $|\nu|$ is Fock-Carleson. Then, $T_\nu^t \in \mathcal C_1$.
\end{thm}

We denote by $\mathcal M = \mathcal M(\operatorname{BUC}(\mathbb C^n))$ the maximal ideal space of the commutative unital $C^\ast$-algebra $\operatorname{BUC}(\mathbb C^n)$, which we understand as a compactification of $\mathbb C^n$. Then, as is described in \cite{Fulsche2020, Fulsche2024, Fulsche_Hagger}, for every $A \in \mathcal C_1$ the map
\begin{align*}
\mathbb C^n \ni z \mapsto \alpha_z(A) \in \mathcal C_1
\end{align*}
extends to a map
\begin{align*}
\mathcal M \ni x \mapsto \alpha_x(A) \in \mathcal C_1
\end{align*}
which is continuous with respect to SOT$^\ast$, i.e.~the map is continuous in strong operator topology for both the operator and the adjoint. The operators $\alpha_x(A)$ for $x \in \mathcal M \setminus \mathbb C^n$ are usually referred to as the \emph{limit operators of $A$}. One then has the following result from \cite{Fulsche_Hagger}, cf.~also \cite{Fulsche2024} which puts it more into the perspective of the algebra $\mathcal C_1$. Here, we denote by $\sigma(B)$ the spectrum of the operator $B$.
\begin{thm}[{\cite[Corollary 29]{Fulsche_Hagger}}]
Let $A \in \mathcal C_1$. Then, the following holds true:
\begin{align*}
\sigma_{ess}(A) = \bigcup_{x \in \mathcal M \setminus \mathbb C^n} \sigma(\alpha_x(A)).
\end{align*}
\end{thm}
We now return to the problem of essential positivity on the Fock space. Let $A \in \mathcal C_1$ be self-adjoint. Then, $\alpha_z(A) = W_zAW_{-z}$ is self-adjoint for any $z \in \mathbb C^n$. Since for each $f \in F_t^2$ the map $z \mapsto \langle \alpha_z(A) f, f\rangle$ is real valued and extends continuously to
\begin{align*}
\mathcal M \ni x \mapsto \langle \alpha_x(A) f, f\rangle,
\end{align*}
we obtain that $\alpha_x(A)$ is self-adjoint for every $x \in \mathcal M$. Now, by the above theorem, we obtain:
\begin{cor}
Let $A \in \mathcal C_1$ be self-adjoint. Then, $A$ is essentially positive if and only if $\alpha_x(A) \geq 0$ for every $x \in \mathcal M \setminus \mathbb C^n$.
\end{cor}
This characterizes essential positivity in largest generality for elements from the Toeplitz algebra. Let us reformulate the result in one particular case, which involves the notion of \emph{vanishing oscillation}. We recall that a continuous bounded function $f: \mathbb C^n \to \mathbb C$ is said to be of vanishing oscillation (at infinity), abbreviated as $f \in \operatorname{VO}_\partial(\mathbb C^n)$, if:
\begin{align*}
\sup_{|w| \leq 1} |f(z) - f(z-w)| \to 0, \quad |z| \to \infty.
\end{align*}
The key fact we will use is the following:
\begin{lem}
Let $A \in \mathcal C_1$. Then, $\widetilde{A}$ is of vanishing oscillation if and only if for every $x \in \mathcal M \setminus \mathbb C^n$ there exists some $c_x \in \mathbb C$ such that $\alpha_x(A) = c_x I$.
\end{lem}
\begin{proof}
Using the correspondence theorem, \cite[Theorem 2.21]{Fulsche2020} or \cite[Theorem 3.1]{Fulsche2024}, together with an adaptation of \cite[Theorem 36]{Hagger2019}, shows that for $A \in \mathcal C_1$, $\widetilde{A} \in \operatorname{VO}_\partial (\mathbb C^n)$ is equivalent to $\alpha_x(A)$ being a constant multiple of the identity for every $x \in \mathcal M \setminus \mathbb C^n$. 
\end{proof}
\begin{cor}\label{cor:vo}
Let $A \in \mathcal C_1$ be self-adjoint such that $\widetilde{A}$ is of vanishing oscillation. Then, $A$ is essentially positive if and only if $\liminf_{|z|\to \infty} \widetilde{A}(z) \geq 0$.
\end{cor}
\begin{proof}
By the assumption, $\alpha_x(A) = c_x I$ for every $x \in \mathcal M \setminus \mathbb C^n$. Since $c_x = \langle \alpha_x(A) 1, 1\rangle$ continuously depends on $x \in \mathcal M$, it is not hard to see that
\begin{align*}
\inf_{x \in \mathcal M \setminus \mathbb C^n} c_x &= \inf_{x \in \mathcal M\setminus \mathbb C^n} \langle \alpha_x(A) 1, 1\rangle \\
&= \liminf_{|z| \to \infty} \langle \alpha_z(A) 1, 1\rangle\\
&= \liminf_{|z|\to\infty} \widetilde{A}(-z).
\end{align*}
In this case, we have
\begin{align*}
\sigma_{ess}(A) = \bigcup_{x \in \mathcal M \setminus \mathbb C^n} \sigma(\alpha_x(A)) = \{ c_x: ~x \in \mathcal M \setminus \mathbb C^n\}.
\end{align*}
These facts show the result.
\end{proof}
We want to emphasize that any Toeplitz operator $T_f^t$ the symbol of which is of vanishing mean oscillation satisfies the previous  result. This class also contains discontinuous symbols.

Specialising the previous result to $A = T_\nu^t$ with $|\nu|$ Fock-Carleson yields a result which is in analogy to \cite[Theorem 9]{Perala_Virtanen2023} without the assumption of radiality. Indeed, if one assumes that $A \in \mathcal C_1$ such that $\lim_{|z|\to \infty} \widetilde{A}(z)$ exists, then this has the rather strong implication that $A \in \mathcal K(\mathcal H) + \mathbb C I$. In particular, for $A$ self-adjoint with $\lim_{|z| \to \infty} \widetilde{A}(z) \geq 0$, this clearly shows essential positivity of $A$. 

It is well-known that there are certain implications between properties of the symbol, essential positivity of the operator and properties of the Berezin transform. We shortly summarize them for completeness. These facts are certainly well-known and can be proven by different means. Just for the fun of it, we sketch a  proof which goes by considerations of limit operators and limit functions.
\begin{prop}
Let $f \in \operatorname{BUC}(\mathbb C^n)$ be real-valued. If $\liminf_{|z| \to \infty} f(z) \geq 0$, then $T_f^t$ is essentially positive.  
\end{prop}
\begin{proof}
Let us define $\alpha_z(f)(w) = f(w-z)$ for $w, z \in \mathbb C^n$. Let $x \in \mathcal M \setminus \mathbb C^n$ and $(z_\gamma) \subset \mathbb C^n$ a net converging in $\mathcal M$ to $x$. Then, it is not hard to see that $\alpha_{z_\gamma}(f)$ converges to some function $\alpha_{x}(f) \in \operatorname{BUC}(\mathbb C^n)$, and the convergence is uniformly on compact subsets (cf.~\cite{Fulsche2020, Fulsche2024} for details). Further, $\alpha_{z_\gamma}(T_f^t) = T_{\alpha_{z_\gamma}(f)}^t \to T_{\alpha_{x}(f)}^t = \alpha_x(T_f^t)$, with convergence in SOT$^\ast$. When $\liminf_{|z|\to \infty} f(z) \geq 0$, then this implies $\alpha_x(f) \geq 0$, hence $\alpha_x(T_f^t) = T_{\alpha_x(f)}^t \geq 0$. Since $x \in \mathcal M \setminus \mathbb C^n$ was arbitrary, this shows that all limit operators are positive.
\end{proof}
Similarly, one proves the following:
\begin{prop}
Let $A \in \mathcal C_1$ be self-adjoint. If $A$ is essentially positive, then $\limsup_{|z| \to \infty} \widetilde{A}(z) \geq 0$.
\end{prop}
We will now prove that the conjecture of Per\"{a}l\"{a} and Virtanen is in general false when the assumption of radiality is dropped. We need the following preparatory lemma.
\begin{lem}\label{lem:nobound}
There exists no constant $M > 0$ such that 
\begin{align*}
\| T_f\|_{ess} \leq M \limsup_{|w| \to \infty} |\widetilde{f}(w)|
\end{align*}
holds true for all $f \in \operatorname{BUC}(\mathbb C^n)$.
\end{lem}
\begin{proof} We first want to emphasize that the proof is a straightforward adaptation of \cite[Corollary 1]{Coburn2012}.
For $z \in \mathbb C^n$ set $h_z^t(w) = e^{2i\im(w \cdot \overline{z})/t}$. Then, $h_z^t \in \operatorname{BUC}(\mathbb C^n)$. As is well-known,
\begin{align*}
T_{h_z^t}^t = e^{-\frac{|z|^2}{2t}}W_z^t,
\end{align*}
therefore
\begin{align*}
\widetilde{T_{h_z}^t}(w) = e^{-\frac{|z|^2}{2t}} \widetilde{W_z^t}(w) = e^{-\frac{3}{2t}|z|^2 + 2i\frac{\im(w \cdot \overline{z})}{t}}
\end{align*}
such that
\begin{align*}
\limsup_{|w| \to \infty} |\widetilde{h_z^t}(w)| = e^{-\frac{3}{2t}|z|^2},
\end{align*}
but, since $e^{\frac{|z|^2}{2t}} T_{h_z^t}^t = W_z^t$ is unitary, we have
\begin{align*}
\| T_{h_z^t}^t\|_{ess} = e^{-\frac{|z|^2}{2t}} \| W_z^t\|_{ess} = e^{-\frac{|z|^2}{2t}}.
\end{align*}
Therefore, any anticipated estimate $\| T_f^t \|_{ess} \leq M \limsup_{|w| \to \infty} |\widetilde{f}(w)|$ is violated by
\begin{align*}
\| T_{h_z^t}\|_{ess} = e^{-\frac{|z|^2}{2t}} \leq M e^{-\frac{3}{2t}|z|^2},
\end{align*}
since the right-hand side goes faster to $0$ than the left-hand side as $|z| \to \infty$.
\end{proof}
\begin{prop}\label{prop:nopos}
Let $f \in \operatorname{BUC}(\mathbb C^n)$ be real-valued. Then, for $T_f^t$ being essentially positive, it is not sufficient that $\liminf_{|w|\to \infty} \widetilde{f}(w) \geq 0$.
\end{prop}
\begin{proof}
The proof is an adaptation of \cite[Theorem 2.3]{Zhao_Zheng2014} for essential positivity.

We assume the contrary, i.e.\ we assume that $\liminf_{|w| \to \infty} \widetilde{f}(w) \geq 0$ would imply essential positivity for every real-valued $f \in \operatorname{BUC}(\mathbb C^n)$. Consider such real-valued $f$. Then, letting $C := \limsup_{|w|\to \infty} |\widetilde{f}(w)| \in [0, \infty)$, we have
\begin{align*}
\liminf_{|w|\to \infty} (C \mp f)^{\sim}(w) \geq 0.
\end{align*}
Here, $(C \mp f)^{\sim}$ denotes the Berezin transform of the functions $C \mp f$. By assumption, we therefore obtain that $T_{C \mp f}^t = CI \mp T_f^t$ is essentially positive, hence:
\begin{align*}
\sigma_{ess}(T_f^t) \subset [-C, C].
\end{align*}
Since $T_f^t$ is self-adjoint (as $f$ is assumed real-valued) and clearly it holds true that $\limsup_{|w|\to \infty} |\widetilde{f}(w)| \leq C$, this yields:
\begin{align*}
\| T_f^t\|_{ess} \leq C = \limsup_{|w|\to \infty} |\widetilde{f}(w)|
\end{align*}
for any real-valued $f \in \operatorname{BUC}(\mathbb C^n)$. If we now let $g \in \operatorname{BUC}(\mathbb C^n)$ arbitrary (i.e.~without the assumption of real-valuedness), we have
\begin{align*}
\widetilde{g} = \widetilde{\re(g)} + \widetilde{i\im(g)} = \re(\widetilde{g}) + i\im(\widetilde{g}),
\end{align*}
which gives
\begin{align*}
\limsup_{|w|\to \infty} |\widetilde{\re(g)}(w)|, ~\limsup_{|w|\to \infty} |\widetilde{\im(g)}(w)| \leq \limsup_{|w|\to \infty} |\widetilde{g}(w)|.
\end{align*}
This now implies:
\begin{align*}
\| T_g^t\|_{ess} &\leq \| T_{\re(g)}^t\|_{ess} + \| T_{\im(g)}^t\|_{ess}\\
&\leq \limsup_{|w|\to \infty} |\widetilde{\re(g)}(w)| + \limsup_{|w|\to \infty} |\widetilde{\im(g)}(w)|\\
&\leq 2\limsup_{|w|\to \infty} |\widetilde{g}(w)|.
\end{align*}
Since this would hold true for any $g \in \operatorname{BUC}(\mathbb C^n)$, this violates the previous lemma.
\end{proof}
In light of this result, it would be surprising if the conjecture of Per\"{a}l\"{a} and Virtanen would turn out to hold true: Our discussion has shown that it is not the assumption of radiality, but the implicit assumption of vanishing oscillation which yields their result \cite[Theorem 9]{Perala_Virtanen2023}. Nevertheless, we cannot entirely rule out the possibility that some strange effects show up under the additional assumption of radiality. Inspecting the previous arguments, it turns out that to disprove the conjecture, one needs to find a sequence of radial functions $f_k \in L^\infty(\mathbb C^n)$ such that 
\begin{align*}
\frac{\| T_{f_n}^t\|_{ess}}{\limsup_{|w|\to \infty} \widetilde{f_n}(w)} \to 0, \quad n \to \infty.
\end{align*}
This would then imply a version of Lemma \ref{lem:nobound} for radial symbols, which in turn would (by arguments analogous to the previous result) show that the conjecture is false.

We consider the special case $n = 1$ and $t = 2$. When diagonalizing the radial operator by the standard orthonormal basis, the task of finding such radial functions $f_k \in L^\infty(\mathbb C)$ is equivalent to the problem of finding a sequence of functions $f_k \in L^\infty([0, \infty))$ such that we have
\begin{align*}
\frac{\limsup_{m \to \infty} \frac{1}{m! 2^m} \left | a_{m,k} \right |}{\limsup_{s \to \infty}  \left| \sum_{\nu = 0}^\infty \frac{s^{2\nu}}{(\nu! 2^\nu)^2}e^{-\frac{s^2}{2}} a_{m,\nu} \right| } \overset{k \to \infty}{\longrightarrow} 0,
\end{align*}
where
\begin{align*}
a_{m,k} = \int_0^\infty f_k(r) e^{-\frac{r^2}{2}} r^{2m+1}dr.
\end{align*}
\begin{rem}
The methods of limit operators is also available for the Bergman space of the disc or even bounded symmetric domains, cf.~\cite{Hagger2017, Hagger2019}. There, the algebra $\mathcal C_1$ has to be replaced by the algebra of band-dominated operators. One obtains that a self-adjoint band-dominated operator $A$ is essentially positive if and only if all of its limit operators are positive. Improvements of this result are available, in similar ways, when the Berezin transform of the operator is of vanishing oscillation, cf.\ \cite[Theorem 36]{Hagger2019}.

Further, note that reasoning similar to Lemma \ref{lem:nobound} and Proposition \ref{prop:nopos} work on the Bergman space. Hence, without the assumption of radiality, the conjecture of Per\"{a}l\"{a} and Virtanen is wrong also on the Bergman space.
\end{rem}

\bibliographystyle{amsplain}
\bibliography{References}

\providecommand{\bysame}{\leavevmode\hbox to3em{\hrulefill}\thinspace}
\providecommand{\MR}{\relax\ifhmode\unskip\space\fi MR }
\providecommand{\MRhref}[2]{%
  \href{http://www.ams.org/mathscinet-getitem?mr=#1}{#2}
}
\providecommand{\href}[2]{#2}
\begin{thebibliography}{10}

\bibitem{Bauer_Isralowitz2012}
W.~Bauer and J.~Isralowitz, \emph{{C}ompactness characterization of operators
  in the {T}oeplitz algebra of the {F}ock space {$F_\alpha^p$}}, J. Funct.
  Anal. \textbf{263} (2012), 1323--1355.

\bibitem{Coburn2012}
L.~A. Coburn, \emph{Berezin transform and {W}eyl-type unitary operators on the
  {B}ergman space}, Proc. Amer. Math. Soc. \textbf{140} (2012), 3445–3451.

\bibitem{Fulsche2020}
R.~Fulsche, \emph{{Correspondence theory on $p$-Fock spaces with applications
  to Toeplitz algebras}}, J. Funct. Anal. \textbf{279} (2020), 108661.

\bibitem{Fulsche2024}
\bysame, \emph{Toeplitz operators on non-reflexive {F}ock spaces}, Rev. Mat.
  Iberoam. \textbf{40} (2024), 1115–1148.

\bibitem{Fulsche_Hagger}
R.~Fulsche and R.~Hagger, \emph{{Fredholmness of Toeplitz operators on the Fock
  space}}, Complex Anal. Oper. Theory \textbf{13} (2019), 375--403.

\bibitem{Hagger2017}
R.~Hagger, \emph{{The essential spectrum of Toeplitz operators on the unit
  ball}}, Integr. Equ. Oper. Theory \textbf{89} (2017), 519--556.

\bibitem{Hagger2019}
\bysame, \emph{{Limit operators, compactness and essential spectra on bounded
  symmetric domains}}, J. Math. Anal. Appl. \textbf{470} (2019), 470--499.

\bibitem{Isralowitz_Zhu2010}
J.~Isralowitz and K.~Zhu, \emph{{Toeplitz operators on the Fock space}},
  Integr. Equ. Oper. Theory \textbf{66} (2010), 593--611.

\bibitem{Perala_Virtanen2023}
A.~Per\"{a}l\"{a} and J.~A. Virtanen, \emph{Essential positivity}, Proc. Amer.
  Math. Soc. \textbf{151} (2023), 4807--4815.

\bibitem{Zhao_Zheng2014}
X.~Zhao and D.~Zheng, \emph{{Positivity of Toeplitz operators via Berezin
  transform}}, J. Math. Anal. Appl. \textbf{416} (2014), 881--900.

\bibitem{Zhu2012}
K.~Zhu, \emph{Analysis on {F}ock {s}paces}, Graduate Texts in Mathematics, vol.
  263, Springer US, New York, 2012.

\end{thebibliography}

\bigskip

\noindent
Robert Fulsche\\
\href{fulsche@math.uni-hannover.de}{\Letter fulsche@math.uni-hannover.de}
\\

\noindent
Institut f\"{u}r Analysis\\
Leibniz Universit\"at Hannover\\
Welfengarten 1\\
30167 Hannover\\
GERMANY

\end{document}